\documentclass[de, en, amspaper, secnum, nobiblatex]{./cls/myamspaper}

\usepackage[top=40mm, bottom=40mm, left=30mm, right=30mm]{geometry}

\DeclareMathAlphabet{\mathpzc}{OT1}{pzc}{m}{it}

\usepackage{marginnote}

\newcommand{\vertiii}[1]{{\left\vert\kern-0.25ex\left\vert\kern-0.25ex\left\vert #1 
    \right\vert\kern-0.25ex\right\vert\kern-0.25ex\right\vert}}

\title[Distal systems in topological dynamics and ergodic 
stheory]{Distal systems in topological dynamics and ergodic theory}

\begin{document}
\begin{abstract}
  We generalize a result of Lindenstrauss on the interplay between 
  measurable and topological dynamics which shows that every separable ergodic measurably 
  distal dynamical system has a minimal distal model. We show that such a model can,
  in fact, be chosen completely canonically.
  The construction is performed by going through the Furstenberg--Zimmer tower of a measurably
  distal system and showing that at each step, there is a simple and canonical
  distal minimal model. This hinges on a new characterization of isometric extensions in 
  topological dynamics.
\end{abstract}

\maketitle

The famous Furstenberg structure theorem shows that every distal 
minimal topological dynamical system can be built from a trivial 
system by only successively performing \enquote{structured} (i.e., 
(pseudo-)isometric) extensions, see, e.g., \cite{Furs1963} and 
\cite[Section V.3]{deVr1993}. As shown by Zimmer 
(see \cite[Theorem 8.7]{Zimm1976}) this classical result has an 
analogue for distal systems in ergodic theory which had beed 
introduced earlier by Parry in \cite{Parr1968}: A measure-preserving 
system on a standard probability space is distal if and only if it can be 
constructed from a tower of measure-theoretically isometric extensions. The 
significance of these measurably distal systems is due to the fact 
that---by the Furstenberg-Zimmer structure theorem---any measure-preserving 
system can be recovered by taking a weakly mixing extension of a distal one 
(see, e.g., \cite{Furs1977}, \cite[Chapter 2]{Tao2009} or \cite{EHK2021}). 
This allows to reduce the proof of important results for measure-preserving 
transformations, such as Furstenberg's recurrence result used for an ergodic 
theoretic proof of Szemer\'{e}di's theorem, to the case of distal systems.

Given a minimal distal system $(K;\phi)$ and an ergodic 
$\phi$-invariant probability measure $\mu$ on $K$, it is not 
hard to prove that $(K, \mu; \phi)$ is also distal as a measure-preserving
system. While this was long known, the converse question, whether a 
distal measure-preserving system admits a distal topological model, 
was only answered much later in an 
article of Lindenstrauss (see \cite{Lindenstrauss1999} and also 
\cite[Sections 5 and 13]{GlWe2004}).
\begin{theorem*}
	Every ergodic distal measure-preserving system on a standard probability 
	space has a minimal distal metrizable topological model. 
\end{theorem*}

Topological models such as these continue to prove useful since they 
allow to use a range of topological results to derive results in 
ergodic theory. For recent examples, we refer to the proof 
of pointwise convergence for multiple averages for distal systems
(see \cite{HuangShaoXiangdong2019}), new approaches to the Furstenberg--Zimmer structure 
theorem (see \cite[Section 7.4]{EHK2021}), or a new
approach to the Host--Kra factors in \cite{GutmanLian2019}.

The proof of Lindenstrauss rests on the Mackey--Zimmer representation of 
isometric extensions as skew-products and on measure-theoretic 
considerations. In this article we pursue an operator theoretic approach to the result and 
even prove that for 
every ergodic distal system there is a canonical minimal distal topological 
model. This allows to construct such models in a functorial way: Every extension between ergodic distal systems induces a topological extension between their canonical models.

The key step is to show that an isometric extension of measure-preserving systems admits 
a completely canonical topological model that is a (pseudo)isometric extension of 
topological dynamical systems. The proof of this requires a new functional-analytic 
characterization of structured extensions in topological dynamics as 
established in \cite{EdKr2020}, related results from ergodic theory (see \cite{EHK2021}), 
as well as a result of Derrien on approximation of measurable cocycles by 
continuous ones (see \cite{Derr2000}). The 
functional-analytic view makes the parallels between the structure theory of 
topological and measure-preserving dynamical systems more apparent and leads to 
canonical models in a straightforward way. In addition, our methods allow us to generalize the result of Lindenstrauss to measure-preserving 
transformations on arbitrary probability spaces, following up on a recent 
endeavor to drop separability assumptions from classical results of ergodic 
theory (see \cite{JaTa2020} and \cite{EHK2021}).

\textbf{Organization of the article.} We start in Section 1 with the 
concepts of structured extensions of topological dynamical systems. Then, based 
on the results of \cite{EdKr2020}, we prove an operator theoretic characterization 
of (pseudo)isometric extensions in terms of the Koopman operator, see 
\cref{fapseudoisometric}. In Section 2 we consider structured extensions in 
ergodic theory, i.e., extensions with relative discrete spectrum, and recall an 
important characterization from \cite{EHK2021}, see \cref{chardiscretespectrum}. 
The major work is done in Section 3 where we construct topological models for 
structured extensions of measure-preserving systems, see \cref{mainthm}. This 
result is applied in the final section of the article to show that every ergodic 
distal measure-preserving system has a minimal distal topological model, see 
\cref{mainthmdistal}. Finally, we discuss the meaning of our result in a category theoretical sense in \cref{remcat}.

\textbf{Preliminaries and notation.} We now set up the notation and recall some 
important concepts from topological dynamics and ergodic theory. We refer to the 
monograph \cite{EFHN2015} as a general reference for the operator theoretic 
approach pursued in this article.

In the following all vector spaces are complex and all compact spaces are assumed 
to be Hausdorff. If $E$ and $F$ are Banach spaces, then $\mathscr{L}(E,F)$ denotes 
the space of all bounded linear operators from $E$ to $F$. We write 
$\mathscr{L}(E) \coloneqq \mathscr{L}(E,E)$ and $E' \coloneqq \mathscr{L}(E,\C)$. 
If $T \in \mathscr{L}(E,F)$ is a bounded operator, then $T' \in \mathscr{L}(F',E')$ 
denotes its adjoint.

 G iven a compact space $K$, we write $\mathcal{U}_K$ for the 
unique uniformity compatible with the topology of $K$. Moreover, $\mathrm{C}(K)$ 
denotes the space of all continuous complex-valued functions which is a unital 
commutative $\mathrm{C}^*$-algebra (cf. \cite[Chapter 4]{EFHN2015}). Using the 
Markov-Riesz representation theorem (see \cite[Appendix E]{EFHN2015}), we identify 
its dual space $\mathrm{C}(K)'$ with the space of all complex regular Borel 
measures on $K$. Likewise, if $\uX = (X, \Sigma, \mu)$ is a probability space, 
then we write $\mathrm{L}^p(\uX)$ with $1 \leq p \leq \infty$ for the associated 
complex $\mathrm{L}^p$-spaces and identify the dual $\mathrm{L}^1(\uX)'$ with 
$\mathrm{L}^\infty(\uX)$.

A \emph{topological dynamical system} $(K;\varphi)$ consists of a compact space 
$K$ and a homeomorphism $\varphi \colon K \rightarrow K$. It is \emph{minimal} if 
there are non-trivial closed subsets $M \subset K$ with $\varphi(M) = M$. Moreover, 
we call $(K;\varphi)$ \emph{metrizable} if the underlying compact space $K$ is 
metrizable. We refer to \cite{Ausl1988} for a general introduction to such systems. 
Topological dynamical systems can be studied effectively via operator theory by 
considering the induced \emph{Koopman operator} $T_\varphi \in \mathscr{L}(\mathrm{C}(K))$ 
defined by $T_\varphi f \coloneqq f \circ \varphi$ for $f \in \mathrm{C}(K)$, see 
\cite[Chapter 4]{EFHN2015}. We remind the reader that $T_\varphi$ is a *-automorphism of 
the $\mathrm{C}^*$-algebra $\mathrm{C}(K)$. In fact, every $*$-automorphism of 
$\mathrm{C}(K)$ is a Koopman operator associated to a uniquely determined homeomorphism 
of $K$ (see \cite[Theorem 4.13]{EFHN2015}). We write 
$\mathrm{P}_{\varphi}(K) \subset \mathrm{C}(K)'$ for the space of all invariant 
probability measures $\mu$ on $K$, i.e., $T_\varphi'\mu = \mu$. Moreover, 
$\supp \mu$ denotes the support of such a measure (see \cite[page 82]{EFHN2015}), 
and we say that $\mu$ is \emph{fully supported} if $\supp \mu = K$. Moreover, we 
write $(K,\mu)$ for the induced probability space.

Classically, a \emph{measure-preserving point transformation} is a pair 
$(\uX;\varphi)$ of a probability space $\uX = (X,\Sigma_X,\mu_X)$ and a measurable 
and measure-preserving map $\varphi \colon \uX \rightarrow \uX$ which is 
\emph{essentially invertible}, i.e., there is a map $\psi \colon \uX \rightarrow \uX$ 
such that $\psi \circ \varphi = \mathrm{id}_X = \varphi \circ \psi$ almost everywhere. 
We refer to \cite{Glas} and \cite{EiWa2011} for an introduction. Given any 
measure-preserving point transformation $(\uX;\varphi)$ we define the 
\emph{Koopman operator} on the corresponding $\mathrm{L}^1$-space via 
$T_\varphi f \coloneqq f \circ \varphi$ for $f \in \mathrm{L}^1(\uX)$. These 
operators are so-called \emph{Markov lattice isomorphisms} on the Banach lattice 
$\mathrm{L}^1(\uX)$, i.e., invertible isometries $T \in \mathscr{L}(\mathrm{L}^1(\uX))$ 
satisfying
\begin{itemize}
	\item $|Tf| = T|f|$ for every $f \in \mathrm{L}^1(\uX)$, and
	\item  $T\mathbbm{1} = \mathbbm{1}$. 
\end{itemize}
We refer to \cite[Chapter 13]{EFHN2015} for more information on such operators.
If $\uX$ is a standard probability space (see \cite[Definition 6.8]{EFHN2015}), 
then a result of von Neumann shows that every Markov lattice isomorphism 
$T \in \mathscr{L}(\mathrm{L}^1(\uX))$ is a Koopman operator of a measure-preserving 
point transformation  (see \cite[Proposition 7.19 and Theorem 7.20]{EFHN2015}). 
For a general probability space $\uX$ one can only show that such operators are 
induced by transformations of the measure algebra of $\uX$ 
(see \cite[Theorem 12.10]{EFHN2015}). Here, we avoid these measure-theoretic intricacies by defining a 
measure-preserving system in terms of operators theory. A \emph{measure-preserving system} 
is a pair $(\uX;T)$ of a probability space $\uX$ and a Markov lattice isomorphism 
$T \in \mathscr{L}(\mathrm{L}^1(\uX))$ (cf. \cite[Definition 12.18]{EFHN2015}). 
It is \emph{ergodic} if the \emph{fixed space}
\begin{align*}
	\fix(T) \coloneqq \{f \in \mathrm{L}^1(X)\mid Tf = f\}
\end{align*}
is one-dimensional (cf. \cite[Proposition 7.15]{EFHN2015}).
We say that $(\uX;T)$ is \emph{separable} if the measure space $\uX$ is separable, 
or equivalently, if the Banach space $\mathrm{L}^1(\uX)$ is separable.

\section{Structured extensions in topological dynamics}

In order to state and prove one of our main results, \cref{mainthm},
we need to briefly recap 
the notions of structured extensions in topological dynamics 
and ergodic theory together with their different characterizations. 
So we start with the notion 
of (pseudo)isometric extensions of topological dynamical systems and their functional-analytic characterization in \cref{fapseudoisometric}. 

\begin{definition}
	An \emph{extension} $q \colon (K;\varphi) \rightarrow (L;\psi)$ between topological 
	dynamical systems $(K;\varphi)$ and $(L;\psi)$ is a continuous surjection 
	$q \colon K \rightarrow L$ such that the diagram
	\begin{align*}
		\xymatrix{
			K \ar[d]_{q} \ar[r]^{\varphi} & K\ar[d]^{q}\\
			L \ar[r]_{\psi} & L
		}
	\end{align*}
  commutes. In this case, we call $(L;\psi)$ a \emph{factor} of $(K;\varphi)$.
	We write $K_l \coloneqq q^{-1}(l)$ for the \emph{fiber of $K$ over $l \in L$} and
	define the \emph{fiber product $K \times_L K$ of $K$ over $L$} as
  \begin{align*}
		K \times_L K \coloneqq \bigcup_{l \in L} K_l \times K_l \subset K \times K.
	\end{align*}
\end{definition}
\begin{remark}\label{fatopextensions}
	There is an equivalent functional-analytic perspective on extensions based on Gelfand duality: Let $(K;\varphi)$ 
	be a topological dynamical system and call a subset $M \subset \mathrm{C}(K)$ 
	\emph{invariant} if $T_\varphi M = M$. If $q \colon (K;\varphi) \rightarrow (L;\psi)$ 
	is an extension, then 
	$T_q \colon \mathrm{C}(L) \rightarrow \mathrm{C}(K), \, f \mapsto f \circ q$ is an 
	isometric *-homorphism intertwining the Koopman operators. Therefore, 
	$A_q \coloneqq T_q(\mathrm{C}(L)) \subset \mathrm{C}(K)$ is an invariant unital 
	$\mathrm{C}^*$-subalgebra of $\mathrm{C}(K)$. On the other hand, if 
	$A \subseteq \mathrm{C}(K)$ is such an invariant unital $\mathrm{C}^*$-subalgebra, 
	then $T_\varphi$ induces a homeomorphism $\psi$ on the Gelfand space $L$ of $A$ and the embedding $A \hookrightarrow \mathrm{C}(K)$ gives rise to an extension $q \colon (K;\varphi) \rightarrow (L;\psi)$ with $A = A_q$ (see  \cite[Chapter 4]{EFHN2015}).
	Thus, instead of looking at factors of a given system $(K;\varphi)$, one can also 
	examine the invariant unital $\mathrm{C}^*$-subalgebras of $\mathrm{C}(K)$.
\end{remark}

We now look at structured extensions of topological dynamical systems. There are 
basically two ways to start from the notion of an isometric or equicontinuous 
system and relativize it to extensions: One is based 
on the existence of invariant (pseudo)metrics while the other generalizes the 
concept of equicontinuity (cf. \cite[Sections V.2 and V.5]{deVr1993} and \cite[Definition 1.15]{EdKr2020}).

\begin{definition}
	An extension $q \colon (K;\varphi) \rightarrow (L;\psi)$ of topological dynamical 
	systems is called
	\begin{enumerate}[(i)]
		\item \emph{pseudoisometric} if there is a familiy $P$ of continuous mappings
  	\begin{align*}
			p \colon K \times_L K \rightarrow \R_{\geq 0}
		\end{align*}
		such that 
		\begin{itemize}
			\item $p|_{K_l \times K_l}$ is a pseudometric on $K_l$ for every $l \in L$ 
			and $p \in P$,
			\item $\{p|_{K_l \times K_l}\mid l \in L\}$ generates the topology on $K_l$ 
			for every $l \in L$, and
			\item $p(\varphi(x),\varphi(y)) = p(x,y)$ for all $(x,y) \in K\times_L K$ 
			and $p \in P$.
		\end{itemize}
		\item \emph{isometric} if it is pseudoisometric and $P$ in (i) can be chosen 
		to only have one element (which then defines a metric on every fiber).
		\item \emph{equicontinuous} if for every entourage $V \in \mathcal{U}_K$ 
		there is an entourage $U \in \mathcal{U}_K$ such that 
		for every pair $(x,y) \in K \times_L K$
		\begin{align*}
			(x,y) \in U \, \Rightarrow \, (\varphi^k(x),\varphi^k(y)) \in V 
			\textrm{ for every } k \in \Z.
		\end{align*}
	\end{enumerate}	
\end{definition}

\begin{remark}
	Every pseudoisometric extension is equicontinuous by \cite[Proposition 1.17]{EdKr2020} 
	while the converse may fail (see \cite[Example 3.15]{EdKr2020}). However, if 
	$(K;\varphi)$ (and hence also $(L;\psi)$) is minimal, then the two notions coincide 
	(see \cite[Corollary 5.10]{deVr1993}).
\end{remark}

The following is a standard example of an isometric extension.

\begin{example}[Skew rotation]\label{exampleisometricex}
	Let $\T \coloneqq \{x \in \C \mid |x| = 1\}$ and $a \in \T$. We consider $(K;\varphi)$ 
	defined by $K\coloneqq \T^2$ with $\varphi(x,y) \coloneqq (ax,xy)$ for $(x,y) \in K$, 
	and $(L;\psi)$ given by $L \coloneqq \T$ with $\psi(x) = ax$ for $x \in L$. Then the 
	projection $q \colon \T^2 \rightarrow \T$ onto the first component defines an isometric 
	extension $q\colon (K;\varphi) \rightarrow (L;\psi)$.
\end{example}

Recall that a system $(K;\varphi)$ is equicontinuous (see \cite[Chapter 2]{Ausl1988}) 
if and only if the induced Koopman operator $T_\varphi \in \mathscr{L}(\mathrm{C}(K))$ 
has discrete spectrum, i.e., $\mathrm{C}(K)$ is the closed linear hull of all eigenspaces 
of the Koopman operator (see, e.g., \cite[Proposition 1.6]{Edeko2019}). Is there a
more general version of this that can be used to characterize when an extension
$q\colon (K; \phi) \to (L; \psi)$
is pseudoisometric? If $(L; \psi)$ satisfies a mild irreducibility condition,
\cref{fapseudoisometric} below provides an affirmative answer. To state it, we 
require the definition of \emph{topological ergodicity} (in analogy to ergodicity)
and we need to recall the module structure an extension gives rise to.

\begin{definition}
	A topological dynamical system $(K;\varphi)$ is \emph{topologically ergodic} if the 
	\emph{fixed space}
	\begin{align*}
  	\fix(T_\varphi) \coloneqq \{f \in \mathrm{C}(K) \mid T_\varphi f = f\}
	\end{align*}
	of the Koopman operator $T_\varphi \in \mathscr{L}(\mathrm{C}(K))$ is one-dimensional.
\end{definition}
Every minimal system is topologically ergodic, but the class of topologically 
ergodic systems is considerably larger and contains, e.g., all topologically transitive 
systems.

\begin{remark}\label{modulerem1}
  One of the key tools in the study of extensions of dynamical systems 
  is the module structure that canonically emerges from an extension and 
  is often tacitly used. We briefly
  recall this: 
	Let $q \colon K \rightarrow L$ be a continuous surjection between compact spaces 
	and $T_q \colon \mathrm{C}(L) \rightarrow \mathrm{C}(K), \, f \mapsto f \circ q$ 
	the induced isometric $^*$-homomorphism. Via this embedding we can define a 
	multiplication
	\begin{align*}
  	\mathrm{C}(L) \times \mathrm{C}(K) \rightarrow \mathrm{C}(K), 
  	\quad (f,g) \mapsto T_qf  \cdot g
	\end{align*}
  that turns $\mathrm{C}(K)$ into a $\mathrm{C}(L)$-module in a canonical way.
\end{remark}

We now obtain the following operator theoretic charaterization of pseudoisometric 
extensions in terms of a relative notion of discrete spectrum. Recall here that a 
module $M$ over a unital commutative ring $R$ is \emph{projective} if there is another 
module $N$ over $R$ such that the direct sum $M \oplus N$ is free, i.e., has a basis.

\begin{theorem}\label{fapseudoisometric}
  Let $q \colon (K;\varphi) \rightarrow (L;\psi)$ be an extension of topological 
  dynamical systems. Assume that $(L;\psi)$ is topologically ergodic and $q$ is open.	 
	Then the following assertions are equivalent.
	\begin{enumerate}[(a)]
		\item $q$ is pseudoisometric.
		\item The union of all closed, invariant, finitely generated, projective 
		$\mathrm{C}(L)$-submodules is dense in $\mathrm{C}(K)$.
		\item The unital $\mathrm{C}^*$-algebra generated by all closed, invariant, 
		finitely generated, projective $\mathrm{C}(L)$-submodules is the whole space 
		$\mathrm{C}(K)$.
	\end{enumerate}
	If $K$ is even metrizable, then (a) can be replaced by
	\begin{enumerate}[(a')]
		\item $q$ is isometric.
	\end{enumerate}
	If $(L;\psi)$ is minimal, the assumption that $q$ is open can be dropped.
\end{theorem}

\begin{remark}
	Loosely speaking, assertions (b) and (c) of \cref{fapseudoisometric} mean that 
	$\mathrm{C}(K)$ is generated by invariant parts which are \enquote{small} relative 
	to $\mathrm{C}(L)$.
\end{remark}

We remark that, except for the last statement about minimal $(L; \psi)$, 
\cref{fapseudoisometric} is a special case of \cite[Theorem 7.2]{EdKr2020}. Thus, we only need 
to prove this additional statement. To do this, we show that, 
in case of a minimal system $(L;\psi)$,  each of the assertions (a) and (c) 
(and consequently also the stronger conditions (b) and (a')) imply that $q$ is open. 

We start by proving that (a) yields that the extension is open. In fact, this 
implication is valid for the more general class of \emph{distal} extensions 
(cf. \cite[Section 3.12]{Bron1979}).

\begin{definition}\label{defdistalext}
	An extension $q\colon (K;\varphi) \rightarrow (L;\psi)$ is \emph{distal} if 
	the following condition is satisfied:  Whenever $(x,y) \in K \times_L K$ and 
	$(\varphi^{n_\alpha})_{\alpha \in A}$ is a net   with $n_\alpha \in \Z$ for $\alpha \in A$ and   
	$\lim_\alpha \varphi^{n_\alpha}(x) = \lim_\alpha \varphi^{n_\alpha}(y)$, then $x=y$. 
\end{definition}
\begin{lemma}\label{distalopen}
	Let $q \colon (K,\varphi) \rightarrow (L,\psi)$ be a distal extension of 
	topological dynamical systems with $(L,\psi)$ minimal. Then $q$ is open.
\end{lemma}
The result is stated as a remark on page 2 of \cite{Ausl2013} without proof. 
Since we have found a proof in the literature only for the case of minimal 
$(K;\varphi)$, we provide a proof of \cref{distalopen} based on the arguments 
of \cite[Lemma 3.14.5]{Bron1979} and \cite[Theorem 10.8]{Ausl1988}. 
\begin{proof}[Proof of \cref{distalopen}]
	Consider the Ellis semigroup of the system $(K;\varphi)$ given by 
  \begin{align*}
	  \mathrm{E}(K;\varphi) \coloneqq \overline{\{\varphi^n \mid n \in  \Z \}} \subseteq K^K
  \end{align*}
	where the closure is taken with respect to the topology of pointwise convergence, see 
	\cite[Chapter 3]{Ausl1988}. For every $l \in L$ the set
	\begin{align*}
		E_l \coloneqq \{\vartheta \colon K_l \rightarrow K_l \mid 
		\exists \tau \in \mathrm{E}(K;\varphi) \textrm{ with } 
		\vartheta = \tau|_{K_l}\} \subset K_l^{K_l}
	\end{align*}
	eqipped with composition of mappings and the product topology is a compact 
	right-topological semigroup (see \cite[Section 1.3]{BeJuMi1989} or 
	\cite[Chapter 16]{EFHN2015} for this concept). 
	As a preliminary step, we show that these are actually groups. 
	To this end, we recall from the structure theory of compact right-topological
	semigroups that every such semigroup contains at least one idempotent and 
	is a group if and only if it has a unique idempotent 
	(see \cite[Theorem 2.12 and Theorem 3.11]{BeJuMi1989}).
	
	Now take 
	$l \in L$ and an idempotent $\vartheta \in E_l$, i.e, $\vartheta^2 = \vartheta$. 
	For $x \in K_l$ consider $y \coloneqq \vartheta(x) \in K_l$. Then 
	$\vartheta(y) = \vartheta^2(x) = \vartheta(x)$ which implies $x = y$ since 
	$q$ is distal. Therefore, $\mathrm{id}_{K_l}$ is the only idempotent in $E_l$ and thus $E_l$ is 
	in fact a group.
	
	We now prove that $q$ is open. Take an $x \in K$ and let $l \coloneqq q(x)$. 
	Assume that $(l_\alpha)_{\alpha \in A}$ is a net in $L$ converging to $l$. It 
	suffices to show that there is a subnet $(l_\beta)_{\beta \in B}$ of 
	$(l_\alpha)_{\alpha \in A}$ and $x_\beta \in K_{l_\beta}$ for every $\beta \in B$ 
	such that $\lim_{\beta} x_\beta = x$. We recall that, since $(L;\psi)$ is minimal,
	for every $\alpha \in A$
	\begin{align*}
		\mathrm{E}(L;\psi)(l_\alpha) = \overline{ \{\varphi^n(l_\alpha)\mid n \in \Z\}} = L.
	\end{align*}
	Moreover, 
	\begin{align*}
		\mathrm{E}(K;\varphi) \rightarrow \mathrm{E}(L;\psi), \quad 
		\tau \mapsto [q(y) \mapsto q(\tau(y))]
	\end{align*}
	is a surjective homomorphism of compact right topological semigroups by 
	\cite[Theorem 3.7]{Ausl1988}. With these two observations we find 
	$\tau_\alpha \in \mathrm{E}(K,\varphi)$ with $q(\tau_\alpha(x)) = l_\alpha$ 
	for every $\alpha \in A$. Passing to a subnet, we may assume that 
	$(\tau_\alpha)_{\alpha \in A}$ converges to some $\tau \in \mathrm{E}(K;\varphi)$. 
	Moreover, $\tau(x) = \lim_{\alpha} \tau_\alpha(x) \in K_l$ which already implies 
	$\tau(K_l) \subset K_l$ (see \cite[Lemma 3.12.10]{Bron1979}), i.e., 
	$\vartheta\coloneqq \tau|_{K_l} \in E_l$. But then 
	$x_\alpha \defeq (\tau_\alpha(\vartheta^{-1}(x)))_{\alpha \in A}$ 
	converges to $x$ and $x_\alpha \in K_{l_\alpha}$ for 
	every $\alpha \in A$.
\end{proof}
	
Since equicontinuous, and in particular pseudoisometric extensions, are distal 
(see \cite[Lemma 3.12.5]{Bron1979}), we now obtain that assertion (a) of 
\cref{fapseudoisometric} implies that the extension is open if $(L;\psi)$
is minimal. The following 
result combined with \cref{distalopen} shows that for minimal $(L;\psi)$, 
also (c) implies openness which proves \cref{fapseudoisometric}. In the proof,
we will use the canonical correspondence between Banach bundles and Banach 
modules; the reader can find a self-contained
summary of the essentials in \cite[Section 4]{EdKr2020}.

\begin{lemma}\label{isometricandopen}
	Let $q \colon (K,\varphi) \rightarrow (L,\psi)$ be an extension. Suppose that 
	the $\mathrm{C}^*$-algebra generated by all closed, invariant, finitely 
	generated, projective $\mathrm{C}(L)$-submodules is the whole space 
	$\mathrm{C}(K)$. Then $q$ is equicontinuous.
\end{lemma}
\begin{proof}
	The uniformity on $K$ is generated by the sets $U_{f, \varepsilon}$ that,
	for $f \in \mathrm{C}(K)$ and $\varepsilon >0$, are defined as
	\begin{align*}
		U_{f,\varepsilon} \defeq \{(x,y) \in K \times K\mid |f(x)- f(y)| < \varepsilon\}.
	\end{align*}
	Our assumption therefore yields 
	that $q$ is equicontinuous if and only if the following holds: For every 
	$\uC(L)$-submodule 
	$M \subset \mathrm{C}(K)$ that is closed, invariant, finitely generated, 
	and projective, for every $f \in M$ and every $\varepsilon > 0$, 
	we find an entourage $U \in \mathcal{U}_K$ such that
	\begin{align*}
	  \forall (x,y) \in K \times_L K\colon 
		(x,y) \in U \,  \Rightarrow \,  |f(\varphi^k(x)) -f(\varphi^k(y))|< \varepsilon 
		\quad \textrm{for all } k \in \Z
	\end{align*}
  It suffices to show the claim only 
	for $f \in M$ with $\|f\| \leq 1$. By looking at finitely many generators for 
	$M$, we will in fact be able to show that the uniformity $U$ can be chosen 
	uniformly for all $f\in M$ with $\|f\| \leq 1$. We will do so by constructing a 
	continuous pseudometric $p\colon K \times K \to \R_{\geq 0}$ from $M$ such that
	\begin{equation*}
    \forall f\in M\cap \overline{\uB_1(0)}, \, \forall x, y\in K \colon |f(x) - f(y)| \leq p(x, y).
  \end{equation*}
  The desired uniformity is then given by 
  $U = \{ (x,y)\in K\times K \mid p(x, y) < \epsilon \}$. To do this, we exploit that 
  there is a one-to-one correspondence between projective, finitely generated 
  $\uC(L)$-modules and locally trivial vector bundles over $L$:
	By 
	\cite[Theorem 8.6 and Remarks 8.7]{Gierz1982} and \cite[Example 4.5]{EdKr2020} 
	the vector spaces
	\begin{align*}
		M_l \coloneqq \{f|_{K_l} \mid l \in L\} \subset \mathrm{C}(K_l)
	\end{align*}
	for $l \in L$ define a Banach bundle over $L$ (see \cite{Gierz1982} or 
	\cite{DuGi1983} for this concept). This is locally trivial by 
	\cite[Lemma 4.13]{EdKr2020}, which---using 
	\cite[Proposition 17.2 and Corollary 4.5]{Gierz1982}---can be characterized 
	in the following way:
	\begin{itemize}
		\item There are closed subsets $L_1,\dots, L_m \subset L$ with 
		$L = \bigcup_{j=1}^m L_j$.
		\item For every $n \in \{1, \dots,m\}$ there are
		$s_{n,1}, \dots, s_{n,k_n} \in M$ such that
  	\begin{align*}
			\Phi_n \colon \mathrm{C}(L_n)^{k_n} \rightarrow M|_{q^{-1}(L)}, \quad 
			(f_1, \dots, f_{k_n}) \mapsto \sum_{j=1}^{k_n} f_j s_{n,j}|_{q^{-1}(L_n)}
		\end{align*}
		is a $\mathrm{C}(L_n)$-linear (not necessarily isometric) isomorphism between 
		the product Banach space $\mathrm{C}(L_n)^{k_n}$ with the maximum norm and the 
		subspace $M|_{q^{-1}(L_n)} \subset \mathrm{C}(q^{-1}(L_n))$.
	\end{itemize}
	For every $n \in \{1, \dots, m\}$ we now consider the continuous seminorm
	\begin{align*}
		p_n \colon K \times K \rightarrow \R_{\geq 0}, \quad 
		(x,y) \mapsto \sum_{j=1}^{k_n} |s_{n,j}(x)-s_{n,j}(y)|
	\end{align*}
	and show that there is a constant $C > 0$ such that
	\begin{align*}
		|f(x)- f(y)| \leq C \cdot \max_{n=1,\dots,m} p_n(x,y)
	\end{align*}
	for all $(x,y) \in K \times_L K$ and $f \in M$ with $\|f\| \leq 1$. 
	This will finish the proof.
			
	It suffices to show that for every $n \in \{1, \dots,m\}$ there is a constant 
	$C_n> 0$ such that the inequality
	\begin{align*}
		|f(x)- f(y)| \leq C_n  \cdot p_n(x,y)
	\end{align*}
	holds for every pair $(x,y) \in K \times_L K$ with $q(x) = q(y) \in L_n$ and 
	every $f \in M$ with $\|f\| \leq 1$. We fix $n \in \{1, \dots,m\}$ and set 
	$C_n \coloneqq \|\Phi_n^{-1}\|>0$. For  $(x,y) \in K \times_L K$ with 
	$l \coloneqq q(x) = q(y) \in L_n$ we then obtain for all $f_1, \dots, f_{k_n}
	\in \uC(L_n)$
	\begin{align*}
		\left|\Phi(f_1, \dots, f_{k_n})(x) - \Phi(f_1, \dots, f_{k_n})(y)\right| 
		&\leq  \sum_{j=1}^{k_n} |f_{j}(l)(s_{n,j}(x)-s_{n,j}(y))|\\
		&\leq  \|(f_1, \dots, f_{k_n})\| \cdot p_n(x,y).
	\end{align*}
	Since $\Phi_n$ is an isomorphism, 
	we conclude that
	\begin{align*}
		|f(x) - f(y)| \leq  C_n \cdot p_n(x,y)
	\end{align*}
  for every pair $(x,y) \in K \times_L K$ with $q(x) = q(y) \in L_n$ and every 
  $f \in M$ with $\|f\| \leq 1$. This is the desired inequality.
\end{proof}

\section{Structured extensions in ergodic theory}

We now turn to structured extensions of measure-preserving systems. In our 
operator theoretic language the following is the natural definition of an extension 
in ergodic theory. Recall here that if $\uX$ and $\uY$ are probability spaces, then 
an isometry $J \in \mathscr{L}(\mathrm{L}^1(\uY),\mathrm{L}^1(\uX))$ is a 
\emph{Markov embedding} (or \emph{Markov lattice homomorphism}) if
\begin{itemize}
	\item $|Jf| = J|f|$ for every $f \in \mathrm{L}^1(\uY)$,
	\item $J\mathbbm{1} = \mathbbm{1}$.
\end{itemize}
\begin{definition}
	An \emph{extension} (or \emph{morphism}) $J \colon (\uY;S) \rightarrow (\uX;T)$ 
	of measure-preserving systems is a Markov embedding 
	$J \in \mathscr{L}(\mathrm{L}^1(\uY), \mathrm{L}^1(\uX))$ such that the diagram
	\begin{align*}
		\xymatrix{
  		\mathrm{L}^1(\uX) \ar[r]^{T}  & \mathrm{L}^1(\uX) \\
			\mathrm{L}^1(\uY) \ar[u]^{J} \ar[r]_{S} & \mathrm{L}^1(\uY)\ar[u]_{J}
		}
	\end{align*}
	commutes. If $J$ is also bijective, then it is an \emph{isomorphism} of 
	measure-preserving systems.
\end{definition}
\begin{remark}
	Given any Markov lattice homomorphism 
	$J \in \mathscr{L}(\mathrm{L}^1(\uY),\mathrm{L}^1(\uX))$ for probability 
	spaces $\uX$ and $\uY$, the adjoint 
	$J' \in \mathscr{L}(\mathrm{L}^\infty(\uX),\mathrm{L}^\infty(\uY))$ extends 
	uniquely to a bounded positive operator 
	$\mathbb{E}_{\uY} \in \mathscr{L}(\mathrm{L}^1(\uX),\mathrm{L}^1(\uY))$ satisfying 
	$\mathbb{E}_{\uY}((Jf) \cdot g) = f \cdot \mathbb{E}_{\uY}(g)$ for all 
	$f \in \mathrm{L}^\infty(\uY)$ and $g \in \mathrm{L}^1(\uX)$  
	(see \cite[Section 13.3]{EFHN2015}). We call $\mathbb{E}_{\uY}$ the 
	\emph{conditional expectation operator} associated with $J$. If $J$ is an 
	extension of measure-preserving systems, then $\mathbb{E}_{\uY}$ intertwines 
	the dynamics.
\end{remark}
\begin{remark}\label{modulestructure2}
	As in the topological setting, we obtain a canonical module structure 
	(cf. \cref{modulerem1}). Indeed, if $f_1,f_2 \in \mathrm{L}^\infty(\uY)$, 
	then $J(f_1\cdot f_2) = J(f_1) \cdot J(f_2)$ (see \cite[Section 13.2]{EFHN2015}) 
	and this implies that the multiplication
	\begin{align*}
		\mathrm{L}^\infty(\uY) \times \mathrm{L}^1(\uX) \rightarrow \mathrm{L}^1(\uX), \quad 
		(f,g) \mapsto J(f)\cdot g
	\end{align*}
	turns $\mathrm{L}^1(\uX)$ into a module over $\mathrm{L}^\infty(\uY)$. 
	For every $p \in [1,\infty]$ the space $\mathrm{L}^p(\uX)$ is a 
	$\mathrm{L}^\infty(\uX)$-submodule of $\mathrm{L}^1(\uX)$.
\end{remark}

As in topological dynamics, there are several notions of 
\enquote{structured extensions} in ergodic theory. We use the following one, 
implicitly used by Ellis in \cite{Elli1987} and inspired by the classical 
notion of discrete spectrum for measure-preserving systems. Here, as above, a 
subset $M \subset \mathrm{L}^1(\uX)$ is \emph{invariant} if $T(M) = M$. 
\begin{definition}
	An extension $J \colon (\uY;S) \rightarrow (\uX;T)$ of measure-preserving systems 
	has \emph{relative discrete spectrum} if the union of all finitely generated 
	invariant submodules of $\mathrm{L}^\infty(\uX)$   over $\mathrm{L}^\infty(\uY)$   is dense in $\mathrm{L}^1(\uX)$.
\end{definition}

\begin{example}\label{reldiscrex}
	Consider the skew rotation $q \colon (K; \varphi) \rightarrow (L;\psi)$ of 
	\cref{exampleisometricex}. By equipping the torus $L = \T$ with the Haar 
	measure $\nu$ and its product $K = \T^2$ with the product measure $\mu = \nu \times \nu$, we arrive at an extension $(L,\nu;T_\psi) \rightarrow (K,\mu;T_\varphi)$ of measure-preserving 
	systems which has relative discrete spectrum, see \cite[Example 6.12]{EHK2021}. More	generally, homogenous skew-products are prototypical examples of extensions with 
	relative discrete spectrum (see, e.g., \cite{Zimm1976} and 
	\cite[Sections 4 and 5]{Elli1987}). 
\end{example}

Equivalent definitions of relative discrete spectrum are listed in 
\cite[Proposition 6.13]{EHK2021}. We need one using modules with an orthonormal 
basis which uses the idea that
\begin{align*}
	\mathrm{L}^\infty(\uX) \times \mathrm{L}^\infty(\uX) \rightarrow 
	\mathrm{L}^\infty(\uY), \quad (f,g) \mapsto \mathbb{E}_{\uY}(f\overline{g})
\end{align*}
can be thought of as an $\mathrm{L}^\infty(\uY)$-valued inner product. We refer to 
the article \cite{EHK2021} for a systematic approach to this idea in 
terms of Hilbert modules.

\begin{definition}\label{orthbasis}
	 Let $J \colon (\uY;S) \rightarrow (\uX;T)$ be an extension of measure-preserving 
	 systems. A finite subset $\{e_1, \dots, e_n\} \subset \uL^\infty(\uX)$ is 
	 \emph{$\uY$-orthonormal} 
	 if $\mathbb{E}_{\uY}(e_i \cdot \overline{e_j}) = 
	 \delta_{ij} \mathbbm{1} \in \mathrm{L}^\infty(\uY)$ for 
	 $i,j \in \{1, \dots,n\}$. In this case we say that $e_1, \dots, e_n$ is a 
	 \emph{$\uY$-orthonormal basis} of the  $\mathrm{L}^\infty(\uY)$ -module generated 
	 by $e_1, \dots, e_n$.
\end{definition}
\begin{remark}\label{proporth}
	If $\{e_1, \dots, e_n\} \subset \uL^\infty(\uX)$ is a 
	$\uY$-orthonormal basis of an 
	$\mathrm{L}^\infty(\uY)$-submodule $M$ as in \cref{orthbasis}, then every 
	$f \in M$ can be written as
	\begin{align*}
		f = \sum_{j=1}^n \mathbb{E}_{\uY}(f \cdot \overline{e_j}) e_j.
	\end{align*}		
	With this observation it is readily checked that any submodule 
	$M \subset \mathrm{L}^\infty(\uX)$ having a $\uY$-orthonormal basis is 
	automatically a free (and in particular, projective) module and closed in 
	$\mathrm{L}^\infty(\uX)$.
\end{remark}

The following result (see \cite[Prop. 8.5, Lem. 8.3, Lem 6.8]{EHK2021} or \cite[Remark 5.16 (1)]{Elli1987}) shows that for extensions of ergodic systems 
with relative discrete spectrum there exist \enquote{many} invariant 
finitely generated submodules with an $\uY$-orthonormal basis.
\begin{proposition}\label{chardiscretespectrum}
	Let $J \colon (\uY;S) \rightarrow (\uX;T)$ be an extension of ergodic measure-preserving systems. Then the following assertions are equivalent.
	\begin{enumerate}[(a)]
		\item $J$ has relative discrete spectrum.
		\item The union of all finitely generated invariant  $\mathrm{L}^\infty(\uY)$- submodules of 
		$\mathrm{L}^\infty(\uX)$ having a $\uY$-orthonormal basis is dense in 
		$\mathrm{L}^1(\uX)$.
	\end{enumerate}
\end{proposition}

\section{Topological models for structured extensions}

Comparing \cref{fapseudoisometric} and \cref{chardiscretespectrum} 
(combined with \cref{proporth}) makes the parallelisms between structured 
extensions in topological dynamics and ergodic theory apparent. We now 
study the relation between both worlds. Recall first that, as in 
\cref{reldiscrex}, we can always construct extensions of measure-preserving 
systems from extensions of topological dynamical systems by picking an 
invariant measure.
\begin{definition}
	Let $q \colon (K;\varphi) \rightarrow (L;\psi)$ be an extension of topological 
	dynamical systems. Moreover, let $\mu \in \mathrm{P}_\varphi(K)$ be an invariant 
	probability measure on $K$ and $q_*\mu \in \mathrm{P}_\psi(L)$ its pushforward, 
	i.e., $q_*\mu = T_q'\mu$. Then the extension
	\begin{align*}
		T_q \colon (L,q_*\mu;T_\psi) \rightarrow (K,\mu;T_\varphi), \quad f \mapsto f \circ q
	\end{align*}
	is the \emph{extension of measure-preserving systems induced by $(q,\mu)$}.
\end{definition}

With the help of \cref{fapseudoisometric} we now immediately obtain the following.
\begin{proposition}\label{inducedrelativediscr}
	Assume that $q \colon (K;\varphi) \rightarrow (L;\psi)$ is an open 
	pseudoisometric extension with a topologically ergodic system $(L;\psi)$. For every 
	$\mu \in \mathrm{P}_\varphi(K)$ the induced extension 
	$T_q \colon (L,q_*\mu;T_\psi) \rightarrow (K,\mu;T_\varphi)$ has relative 
	discrete spectrum.
\end{proposition}
\begin{proof}
	By \cref{fapseudoisometric} the union of all closed, invariant, finitely 
	generated, projective $\mathrm{C}(L)$-submodules is dense in $\mathrm{C}(K)$ and, 
	via the canonical map $\mathrm{C}(K) \rightarrow \mathrm{L}^1(K,\mu)$, also dense in 
	$\mathrm{L}^1(K,\mu)$. However, if $M$ is a finitely generated invariant 
	$\mathrm{C}(L)$ submodule of $\mathrm{C}(K)$ with generators 
	$e_1, \dots,e_n$, then the $\mathrm{L}^\infty(L, q_*\mu )$-submodule of 
	$\mathrm{L}^\infty(K,\mu)$ generated by the canonical images of $e_1, \dots, e_n$ 
	in $\mathrm{L}^\infty(K,\mu)$ is also invariant. This shows the claim.
\end{proof}
In particular, by \cref{distalopen} we can construct extensions with relative 
discrete spectrum from pseudoisometric extensions of minimal topological dynamical 
systems. In the remainder of this section we study the converse situation. 
Given an extension $J \colon (\uY;S) \rightarrow (\uX;T)$ of measure-preserving 
systems with relative discrete spectrum, can we find a pseudosisometric 
\emph{topological model}? In order to make this question precise, we recall 
the following definition (cf. \cite[Section 2.2]{Glas} and \cite[Chapter 12]{EFHN2015}).
\begin{definition}
	Let $J_i \colon (\uY_i;S_i) \rightarrow (\uX_i;T_i)$ be extensions of 
	measure-preserving systems for $i=1,2$. An \emph{isomorphism} from 
	$J_1$ to $J_2$ is a pair $(\Psi,\Phi)$ of an isomorphism 
	$\Psi \colon (\uY_1;S_1) \rightarrow (\uY_2;S_2)$ and an isomorphism 
	$\Phi \colon (\uX_1;T_1) \rightarrow (\uX_2;T_2)$ such that the diagram
	\begin{align*}
		\xymatrix{
			\mathrm{L}^1(\uX_1) \ar[r]^{\Phi}  & \mathrm{L}^1(\uX_2) \\
			\mathrm{L}^1(\uY_1) \ar[u]^{J_1} \ar[r]_{\Psi} & \mathrm{L}^1(\uY_2)\ar[u]_{J_2}
		}
	\end{align*}
  commutes.

	If $J \colon (\uY;S) \rightarrow (\uX;T)$ is an extension of measure-preserving 
	systems, then a \emph{topological model} for $J$ is a pair $(q,\mu;\Psi,\Phi)$ 
	such that
	\begin{itemize}
		\item $q \colon (K;\varphi) \rightarrow (L;\psi)$ is an extension of topological 
		dynamical systems,
		\item $\mu \in \mathrm{P}_\varphi(K)$ is a fully supported invariant probability 
		measure, and
		\item $(\Psi,\Phi)$ is an isomorphism from the extension $T_q$ induced by 
		$(q,\mu)$ to $J$.
	\end{itemize}				
\end{definition}
\begin{remark}\label{remtopmodels}
	We use the following observation to conctruct topological models: If 
	$(q,\mu;\Psi,\Phi)$ is a topological model for $J\colon (\uY; S) \to (\uX; T)$ with 
	$q \colon (K;\varphi) \rightarrow (L;\psi)$, then 
  \begin{align*}
		A_{q,\mu} \coloneqq \Psi(\mathrm{C}(L)) \subset \mathrm{L}^\infty(\uY) 
		\textrm{ and } B_{q,\mu} \coloneqq \Phi(\mathrm{C}(K)) \subset \mathrm{L}^\infty(\uX)
	\end{align*}
	are invariant unital  $\mathrm{C}^*$-subalgebras of $\mathrm{L}^\infty(\uY)$ and 
	$\mathrm{L}^\infty(\uX)$, respectively, being dense in the corresponding 
	$\mathrm{L}^1$-spaces. Moreover, $J(A) \subset B$. Conversely, by Gelfand's  
	representation theory, for every pair $(A,B)$ of $\mathrm{L}^1$-dense invariant 
	unital $\mathrm{C}^*$-subalgebras $A \subset \mathrm{L}^\infty(\uY)$ and 
	$B \subset \mathrm{L}^\infty(\uX)$ with $J(A) \subset B$, we can conctruct, in a canonical way, a topological 
	model $(q,\mu;\Psi,\Phi)$ for $J$ such that  $A = A_{q,\mu}$ and $B = B_{q,\mu}$ , 
	cf. \cite[Chapter 12]{EFHN2015}.
\end{remark}
Using \cref{fapseudoisometric} and \cref{chardiscretespectrum} we 
always find a pseudoisometric topological model for extensions of ergodic measure-preserving systems with relative discrete spectrum.
\begin{theorem}\label{stonemodel}
	Let $J \colon (Y;S) \rightarrow (X;T)$ be an extension  of ergodic measure-persersving systems with relative discrete spectrum . Then $J$ has a topological model 
	$(q,\mu;\Psi;\Phi)$ such that $q \colon (K;\varphi) \rightarrow (L;\psi)$ is an 
	open pseudoisometric extension with   $(K;\varphi)$   topologically ergodic.
\end{theorem}
\begin{proof}
	We define $A \coloneqq \mathrm{L}^\infty(\uY)$ and take $B$ as the unital 
	$\mathrm{C}^*$-algebra generated by all invariant, finitely generated, 
	projective $\mathrm{L}^\infty(\uY)$-submodules of $\mathrm{L}^\infty(\uX)$ 
	which are closed in $\mathrm{L}^\infty(\uX)$. Clearly, $A$ is dense 
	$\mathrm{L}^1(\uY)$, and, in view of \cref{proporth} and \cref{chardiscretespectrum}, 
	$B$ is also dense in $\mathrm{L}^1(\uX)$. By \cref{remtopmodels} we find a 
	topological model $(q,\mu;\Psi,\Phi)$ for $J$ where $q \colon (K;\varphi) \rightarrow (L;\psi)$ is an extension of topological dyanmical systems. We obtain that 
	$\Phi(\mathrm{C}(L)) = A = \mathrm{L}^\infty(\uY)$. Since the system   $(\uX;T)$     
	is ergodic,   $\fix(T_\varphi)$   is one-dimensional and therefore   $(K;\varphi)$   is 
	topologically ergodic. Also, since we have chosen $A$ to be the whole space 
	$\mathrm{L}^\infty(\uY)$, the induced extension $q$ is open by 
	\cite[Corollary 1.9]{Elli1987}. Finally, since $B = \Phi(\mathrm{C}(K))$ 
	we obtain that $\mathrm{C}(K)$ is generated as a unital C*-algebra by all 
	closed, invariant, finitely generated, projective $\mathrm{C}(L)$-submodules. 
	Thus, $q$ is pseudosisometric by \cref{fapseudoisometric}.
\end{proof}
Note that the construction of the topological model in the proof of \cref{stonemodel} is completely canonical. However, it 
is still unsatisfactory in some ways. For example, by relying on the 
\enquote{Stone model} (i.e., the topological model for the algebra $A = \uL^\infty(\uY)$) 
in the proof of \cref{stonemodel}, we cannot find 
metrizable models for extensions between separable probability spaces since 
$\uL^\infty(\uY)$ is only separable if it is finite-dimensional.
A yet more serious problem is that, given two extensions 
$J_1 \colon (\uZ;R) \rightarrow (\uY;S)$ 
and $J_2 \colon (\uY;S) \rightarrow (\uX;T)$ with relative discrete spectrum,
\cref{stonemodel} does not allow to construct pesudoisometric models $q_1$ 
and $q_2$ which \enquote{fit together} since in \cref{stonemodel} the base 
of the constructed extension is the Stone model. This will be essential in the 
conctruction of distal models by means of succcessive extensions. The 
following result fixes these problems, 
at least in certain situations, by allowing to impose that the topological 
model at the bottom of an extension be any specific given topological model
instead of the Stone model.

Recall that a measure 
$\nu \in \mathrm{P}_\psi(L)$ is called \emph{ergodic} if the induced 
measure-preserving system $(L,\nu;T_\psi)$ is ergodic.
\begin{theorem}\label{mainthm}
 	Let $(L;\psi)$ be a minimal topological dynamical system, 
 	$\nu \in \mathrm{P}_\psi(L)$ a fully supported ergodic measure and 
 	$J \colon (L,\nu;T_\psi) \rightarrow (\uX;T)$ an extension   of ergodic systems   with relative 
 	discrete spectrum. Then there are 
	\begin{enumerate}[(i)]
		\item an open pseudoisometric extension $q \colon (K;\varphi) \rightarrow (L;\psi)$,
		\item a fully supported   ergodic   measure $\mu \in \mathrm{P}_\varphi(K)$ 
		with $q_*\mu = \nu$, and
		\item an isomorphism $\Phi \colon (K,\mu;T_\varphi) \rightarrow (\uX;T)$, 
	\end{enumerate}
  such that $(q,\mu;\mathrm{Id},\Phi)$ is a topological model for $J$. 
  Moreover, if $\uX$ is separable and $L$ is metrizable, then $K$ can 
  be (noncanonically) chosen to be metrizable such that $q$ is an isometric
  extension.
\end{theorem}

The challenge in proving \cref{mainthm} is, given an abundance of finitely 
generated invariant $\uL^\infty(L, \nu)$-submodules, to find an abundance 
of finitely generated invariant $\uC(L)$-submodules. It is nontrivial
that this can be done and it is this and only this point that forces us to 
restrict to $\Z$-actions in this article.  The following lemma shows that 
at least for $\Z$-actions, finitely generated invariant $\uL^\infty(L, \nu)$-submodules 
can indeed be approximated by finitely generated invariant $\uC(L)$-submodules.

\begin{lemma}\label{derriencor}
	Let $(L;\psi)$ be a topological dynamical system, $\nu \in \mathrm{P}_\psi(L)$ 
	fully supported and ergodic, and $J \colon (L,\nu;T_\psi) \rightarrow (\uX;T)$ 
	an extension. Let $M \subset \mathrm{L}^\infty(\uX)$ be an invariant 
	$\mathrm{L}^\infty(L,\nu)$-submodule with orthonormal basis 
	$\{e_1, \dots, e_n \}$. For every $\varepsilon > 0$ there is an 
	$(L,\nu)$-orthonormal set $\{d_1, \dots, d_n \} \subset M$ such that 
  \begin{enumerate}[(i)]
		\item the $\mathrm{C}(L)$-submodule generated by $d_1, \dots, d_n$ is 
		invariant (as well as closed in $\uL^\infty(\uX)$ and projective) and
		\item $\|d_i - e_i\|_{\mathrm{L}^1(\uX)} \leq \varepsilon$ for all 
		$i \in \{1, \dots,n\}$.
	\end{enumerate}
\end{lemma}

The proof rests on the following approximation result which is, in essence, due to 
Derrien (see \cite{Derr2000}). It shows that, given a measurable map with values 
in the compact group $\mathrm{U}(n)$ of unitary $n\times n$-matrices,
one can find an arbitrarily close continuous map that is cohomologous  (cf. \cite[Theorem 3.1]{Lindenstrauss1999}) .

\begin{lemma}\label{derrien}
	Let $(L;\psi)$ be a a metrizable topological dynamical system and 
	$\mu \in \mathrm{P}_{\psi}(L)$ fully supported and ergodic. Assume that 
	$F \colon L \rightarrow \mathrm{U}(n)$ is a Borel measurable map. For every 
	$\varepsilon>0$ there is a Borel measurable map $G \colon L \rightarrow \mathrm{U}(n)$ 
	and a continuous map $H \colon L \rightarrow \mathrm{U}(n)$ such that
	\begin{enumerate}[(i)]
		\item $\nu(\{l \in L \mid G(l) \neq \mathrm{Id}\}) \leq \varepsilon$, and
		\item $(G \circ \psi) \cdot F \cdot G^{-1} =H$ almost everywhere.
	\end{enumerate}
\end{lemma}
\begin{proof}
	If $ \nu $ has no atoms, then \cite[Theorem 1.1]{Derr2000} shows the existence of 
	a Borel measurable map $G \colon L \rightarrow \mathrm{U}(n)$ satisfying (ii). 
	However, an inspection of the proof (see the remarks after \cite[Theorem 1.2]{Derr2000}) 
	reveals that for a given $\varepsilon >0$ one can also ensure (i).

	Now assume that $\nu$ has an atom. Then there is a periodic finite orbit of measure 
	one and therefore, since $\nu$ is fully supported, $L$ is a discrete finite space. 
	In particular, every (Borel measurable) map $G \colon L \rightarrow \mathrm{U}(n)$ 
	is continuous and there is nothing to prove.
\end{proof}

\begin{proof}[Proof of \cref{derriencor}]
	As a first step, we reduce the problem to the case of a 
	metrizable space $L$. So let $M \subset \uL^\infty(\uX)$ be a 
	finitely generated invariant $\uL^\infty(L, \nu)$-submodule with 
	orthonormal basis $e_1, \dots, e_n$. Then
 	\begin{align*}
		Te_i = \sum_{j=1}^n  f_{ij} e_j \quad \textrm{for every } i \in \{1, \dots n\},
	\end{align*}
	where $f_{ij} \defeq \mathbb{E}_{(L,\nu)}(T e_i \cdot \overline{e_j}) \in 
	\mathrm{L}^\infty(L,\nu)$ for $i,j \in \{1, \dots, n\}$ (see \cref{proporth}).
	Let $B \subset \mathrm{L}^\infty(L,\nu)$ be the unital invariant C*-subalgebra 
	generated by the coefficients $f_{ij}$ for $i,j \in \{1, \dots n\}$. Then $B$ is 
	separable. Since $\mathrm{C}(L)$ is dense in $\mathrm{L}^1(L,\nu)$, we find a 
	separable invariant C*-subalgebra $A$ of $\mathrm{C}(L)$, the $\mathrm{L}^1$-closure 
	of which contains $B$. The canonical inclusion map $A  \hookrightarrow \mathrm{C}(L)$ 
	induces an extension $q \colon (L;\psi) \rightarrow (M;\vartheta)$ 
	(see \cref{fatopextensions}), i.e., $T_q(\mathrm{C}(M)) = A$. Since $A$ is separable, 
	the space $M$ is metrizable (see \cite[Theorem 4.7]{EFHN2015}). We equip $M$ with 
	the pushforward measure $q_*\nu$. Then $T_q$ defines an extension 
	$T_q \colon ((M,q_*\nu);T_\vartheta) \rightarrow 
	((L,\nu);T_\psi)$. In particular, we obtain a new extension 
	$JT_q \colon ((M,q_*\nu);T_\vartheta) \rightarrow (\uX;T)$ 
	and $\mathrm{L}^1(\uX)$ is thus a $\mathrm{L}^\infty(M,q_*\nu)$-module 
	(cf. \cref{modulestructure2}). By choice of $A$, the 
	$\mathrm{L}^\infty(M, q_*\nu)$-submodule generated by $e_1, \dots, e_n$ is still 
	invariant. Replacing $(L;\psi)$ by $(M;\vartheta)$ and $\nu$ by $q_*\nu$ we may 
	therefore assume that $L$ is metrizable.

	Next, we show that we can pick representatives for the coefficients
	$f_{ij} \in \uL^\infty(L, \nu)$ which define a $\mathrm{U}(n)$-valued function. To that end, note that
	\begin{align*}
	  \sum_{k=1}^n f_{ik} \overline{f_{jk}}
	  = \mathbb{E}_{(L,\nu)}(T e_i \cdot T \overline{e_j})
		= T_\psi \mathbb{E}_{(L,\nu)}(e_i \overline{e_j}) 
		= T_\psi \delta_{ij} \mathbbm{1} 
    = \delta_{ij} \mathbbm{1}.
  \end{align*}
 	Picking suitable representatives for 
 	$f_{ij} \in \mathrm{L}^\infty(\uY)$ for $i,j \in \{1,\dots,n\}$, which we 
 	denote by the same symbol, we therefore obtain a Borel measurable map
	\begin{align*}
  	F \colon L \mapsto \mathrm{U}(n), \quad l \mapsto (f_{ij}(l))_{i,j}
	\end{align*}
	from $L$ to $\mathrm{U}(n)$.
	For $\varepsilon > 0$ set 
	\begin{align*}
		\delta \coloneqq \varepsilon \cdot \left[2n \cdot 
		\max \left\{\|e_i\|_{\mathrm{L}^\infty(\uX)} \mid 
		i \in \{1, \dots,n\}\right\}\right]^{-1}.
	\end{align*}
	Apply \cref{derrien} to find a Borel measurable map 
	\begin{align*}
		G \colon L \rightarrow \mathrm{U}(n), \quad l \mapsto (g_{ij}(l))_{i,j}
	\end{align*}
	and a continuous map 
	\begin{align*}
		H \colon L \rightarrow \mathrm{U}(n), \quad l \mapsto (h_{ij}(l))_{i,j}
	\end{align*}
	such that
	\begin{enumerate}[(i)]
		\item $\nu(\{l \in L \mid G(l) \neq \mathrm{Id}\}) \leq \delta$, and
		\item $(G \circ \psi) \cdot F =H \cdot G$ almost everywhere.
	\end{enumerate}
	Now consider the elements 
	$d_i \coloneqq \sum_{k=1}^n g_{ik} e_k \in M\subset \mathrm{L}^\infty(\uX)$ for 
	$i \in \{1, \dots, n\}$. Since 
	\begin{align*}
	  Td_i 
	  &=  \sum_{k=1}^n (T_\psi g_{ik}) \cdot T e_k 
	  = \sum_{k=1}^n \sum_{j=1}^n  (T_\psi g_{ik}) \cdot f_{kj} \cdot e_j \\
		&=  \sum_{j=1}^n \left(\sum_{k=1}^n  (T_\psi g_{ik}) \cdot f_{kj}\right) e_j 
		= \sum_{j=1}^n \left(\sum_{k=1}^n h_{ik}g_{kj}\right)e_j 
		= \sum_{k=1}^n h_{ik} d_k
	\end{align*}
	for every $i \in \{1, \dots n\}$, the $\mathrm{C}(L)$-submodule of 
	$\mathrm{L}^\infty(\uX)$ generated by $d_1, \dots, d_n$ is invariant. 
	Moreover, a swift computation confirms thats
	$\mathbb{E}_{(L,\nu)}(d_i\overline{d_j}) = \delta_{i,j}\mathbbm{1}$ 
	for $i,j \in \{1, \dots, n\}$ which shows that $\{d_1, \dots, d_n\}$ is an 
	$(L,\nu)$-orthonormal set. In particular, the $\uC(L)$-submodule 
	generated by $d_1, \dots, d_n$ is free and hence closed in $\uL^\infty(\uX)$
	and projective. Finally, since 
	$\nu(\{l \in L \mid G(l) \neq \mathrm{Id}\}) \leq \delta$ 
	we obtain for $i \in \{1, \dots, n\}$
	\begin{align*}
		\|d_i - e_i\|_{\mathrm{L}^1(\uX)}
		&= \left\| \mathbb{E}_{(L, \nu)}(d_i - e_i)\right\|_{\mathrm{L}^1(L, \nu)} 
		= \left\|\sum_{k=1}^n g_{ik} \mathbb{E}_{(L,\nu)}e_k - 
		\mathbb{E}_{(L,\nu)}e_i\right\|_{\mathrm{L}^1(L,\nu)}\\
		&\leq \delta \cdot \left\| \sum_{k=1}^n g_{ik} \mathbb{E}_{(L,\nu)}e_k - 
    \mathbb{E}_{(L,\nu)}e_i\right\|_{\mathrm{L}^\infty(L,\nu)}\\
		&\leq \delta \cdot \left\| \sum_{k=1}^n g_{ik} e_k - e_i\right\|_{\mathrm{L}^\infty(\uX)} 
		\leq \varepsilon.
	\end{align*}
\end{proof}
We can now prove \cref{mainthm}.
\begin{proof}[Proof of \cref{mainthm}]
	Let $B$ be the unital $\mathrm{C}^*$-subalgebra generated by all closed, invariant, 
	finitely generated, projective $\mathrm{C}(L)$-submodules of $\mathrm{L}^\infty(\uX)$. 
	We show that $B$ is dense in $\mathrm{L}^1(\uX)$. Since $J$ has relative discrete 
	spectrum, it suffices to approximate elements $f$ contained in a finitely generated 
	$\mathrm{L}^\infty(L, \nu)$-submodule of $\mathrm{L}^\infty(\uX)$ with an orthonormal 
	basis (see \cref{chardiscretespectrum}). Take an orthonormal basis 
	$\{e_1, \dots, e_n\}$ of such a module $M$. Let $f = \sum_{i=1}^n f_i e_i \in M$ 
	for $f_1, \dots, f_n \in \mathrm{L}^\infty(L, \nu)$ and $\varepsilon > 0$. Set 
	$c_1 \coloneqq \sum_{i=1}^n \|f_i\|_{\mathrm{L}^\infty(L, \nu)} +1 > 0$. Using 
	\cref{derriencor} we find an $(L,\nu)$-orthonormal set $\{d_1, \dots, d_n\} \subset M$ 
	such that its $\mathrm{C}(L)$-linear hull $N$ is invariant and 
  \begin{align*}
		\|d_i - e_i\|_{\uL^1(\uX)} \leq \frac{\varepsilon}{2c_1}
	\end{align*}
	for all $i \in \{1, \dots, n\}$.	Set $c_2 \coloneqq \sum_{i=1}^n \|d_i\|_{\mathrm{L}^\infty(\uX)} +1 > 0$. 
	Since $\mathrm{C}(L)$ is dense in $\mathrm{L}^1(L,\nu)$, we now also find 
	$g_1, \dots, g_n \in \uC(L)$ such that 
	\begin{align*}
		\|f_i - g_i\|_{\uL^1(L, \nu)} \leq \frac{\varepsilon}{2c_2}.
	\end{align*}
	For $g \coloneqq \sum_{i=1}^ng_i d_i \in N$ we then obtain
	\begin{align*}
		\left\|f - g\right\|_{\mathrm{L}^1(\uX)} 
		&\leq \sum_{i=1}^n \|f_i\|_{\mathrm{L}^\infty(L, \nu)} \|e_i - d_i\|_{\mathrm{L}^1(\uX)} 
		+ \sum_{i=1}^n \|f_i - g_i\|_{\mathrm{L}^1(L, \nu)} \|d_i\|_{\mathrm{L}^\infty(\uX)}\\
		&\leq \varepsilon.
	\end{align*}
	Since $N$ is a closed, invariant, finitely generated, projective 
	$\mathrm{C}(L)$-submodule of $\mathrm{L}^\infty(\uX)$ (use \cref{proporth}), 
	this shows that $B$ is dense in $\mathrm{L}^1(\uX)$. By Gelfand theory we now 
	find an extension $q \colon (K;\varphi)\rightarrow (L;\psi)$, an ergodic measure 
	$\mu \in \mathrm{P}_{\phi}(K)$ with $q_*\mu = \nu$ and an isomorphism 
	$\Phi \colon (K,\mu;T_\varphi) \rightarrow (\uX;T)$ with $\Phi(\mathrm{C}(L)) = B$ 
	such that $(q,\mu;\mathrm{Id},\Phi)$ is a topological model for $J$ (see \cref{remtopmodels}). By definition 
	of $B$ and \cref{fapseudoisometric}, the extension $q$ is pseudoisometric. 
				 
	Finally, assume that $\uX$ is separable and $L$ is metrizable. Then we find a 
	sequence $(M_n)_{n \in \N}$ of finitely generated $\mathrm{L}^\infty(L, \nu)$-submodules 
	of $\mathrm{L}^\infty(\uX)$ with an orthonormal basis such that their union is dense 
	in $\mathrm{L}^1(\uX)$. For every $n \in \N$ we find a sequence $(N_{n,k})_{k \in \N}$ 
	of  closed, invariant, finitely generated, projective $\mathrm{C}(L)$-submodules 
	contained in $M_n$ the union of which is dense in $M_n$ with respect to the 
	$\mathrm{L}^1$-norm. Let $B$ the unital $\mathrm{C}^*$-subalgebra of 
	$\mathrm{L}^\infty(\uX)$ generated by $\{N_{n,k}\mid n,k \in \N\}$. 
	Since $\mathrm{C}(L)$ is separable (see \cite[Theorem 4.7]{EFHN2015}), 
	$N_{n,k}$ is separable for all $n,k \in \N$. Therefore $B$ is separable. 
	Proceeding as above yields a pesudoisometric extension 
	$q \colon (K;\varphi)\rightarrow (L;\psi)$, an ergodic measure 
	$\mu \in \mathrm{P}_{\phi}(K)$ with $q_*\mu = \nu$ and an isomorphism 
	$\Phi \colon (K,\mu;T_\varphi) \rightarrow (\uX;T)$ with $\Phi(\mathrm{C}(L)) = B$ 
	such that $(q,\mu;\mathrm{Id},\Phi)$ is a topological model for $J$. Again using 
	\cite[Theorem 4.7]{EFHN2015} we conclude that $K$ is metrizable and therefore 
	$q$ is isometric (see \cref{fapseudoisometric}).
\end{proof}

\section{Topological models for distal systems}

With the help of \cref{mainthm}, we now prove the existence of minimal distal 
topological models for ergodic distal measure-preserving systems. Recall, 
that a topological dynamical system $(K;\varphi)$ is \emph{distal} if the 
extension $q \colon (K;\varphi) \rightarrow (\{\mathrm{pt}\};\mathrm{id})$ 
over a one-point system is distal in the sense of \cref{defdistalext}, i.e., 
if the following condition is satisfied: Whenever $(x,y) \in K \times K$ and 
$(\varphi^{n_\alpha})_{\alpha \in A}$ is a net with 
$\lim_\alpha \varphi^{n_\alpha}(x) = \lim_\alpha \varphi^{n_\alpha}(y)$, then $x=y$. 
	
A typical example of a distal system is the skew torus discussed in 
\cref{exampleisometricex} which is given by an isometric extension of an 
isometric system. Put differently, it can be built from a trivial system 
by performing two isometric extensions. Furstenberg's structure theorem 
extends this observation stating that, in fact, any minimal distal system 
can be built up from a trivial system via successive (pseudo)isometric 
extensions and \emph{projective limits}. We recall the latter concept 
(see also \cite[Section (E.12)]{deVr1993}).
\begin{definition}\label{projective}
	Let $I$ be a directed set. For every $i \in I$ let $(K_i;\varphi_i)$ 
	be a topological dynamical system and for $i \leq j$ let 
	$q_i^j \colon (K_j;\varphi_j) \rightarrow (K_i;\varphi)$ be an extension. 
	Assume that
	\begin{enumerate}[(i)]
  	\item $q_i^j\circ q^k_j = q_i^k$ for all $i \leq j \leq k$, and
		\item $q_i^i = \mathrm{id}_{K_i}$ for every $i \in I$.
	\end{enumerate}
  Then the pair $((K_i;\varphi_i))_{i \in I}, (q_i^j)_{i \leq j})$ is a 
  \emph{projective system}.
		
	A topological dynamical system $(K;\varphi)$ together with extensions 
	$q_i \colon (K;\varphi) \rightarrow (K_i;\varphi_i)$ for every $i \in I$ 
	such that $q_i =  q^j_i \circ q_j$ for all $i \leq j$ is a 
	\emph{projective limit of $((K_i;\varphi_i))_{i \in I}, (q_i^j)_{i \leq j})$} 
	if it satisfies the following universal property:
	\begin{itemize}
		\item Whenever $(\tilde{K};\tilde{\phi})$ is a topological dynamical system and 
		$p_i \colon (\tilde{K};\tilde{\phi}) \rightarrow (K_i;\varphi_i)$ are extensions 
		for every $i \in I$ 
		such that $p_i =  q^j_i \circ p_j$ for all $i \leq j$, then there is a 
		unique extension $q \colon (\tilde{K};\tilde{\phi}) \rightarrow (K;\phi)$ such that 
		the diagram
		\begin{align*}
			\xymatrix{
				(K;\varphi) \ar[d]_{q_i}  & (\tilde{K};\tilde{\phi}) \ar[l]_q \ar[ld]^{p_i}\\
				(K_i;\varphi_i)
			}
		\end{align*}
		commutes for every $i \in I$.
	\end{itemize}
	In this case, we write
	\begin{align*}
 		(K;\varphi) = \lim_{\substack{\longleftarrow\\ i}} (K_i;\varphi_i).
 	\end{align*}
\end{definition}
\begin{remark}\label{remproj}
	Every projective system $((K_i;\varphi_i))_{i \in I}, (q_i^j)_{i \leq j})$ 
	has a projective limit and it is unique up to isomorphy. In fact, we 
	obtain a concrete construction of a projective limit by considering the 
	dynamics on the compact space
	\begin{align*}
		\left\{(x_i)_{i \in I} \in 
		\prod_{i \in I} K_i \mmid \pi_i^j(x_j) = x_i \textrm{ for all } i \leq j\right\}
	\end{align*}
	induced by the product action on $\prod_{i \in I} K_i$ 
	(see \cite[Exercise 2.18]{EFHN2015}). Moreover, if $(K_i;\varphi_i)$ is 
	minimal for every $i \in I$, then also every projective limit of 
	$((K_i;\varphi_i))_{i \in I}, (q_i^j)_{i \leq j})$ is minimal 
	(see \cite[Exercise 3.19]{EFHN2015}).
\end{remark}
\begin{remark}\label{opprojective}
	The following is an operator theoretic view of projective limits. 
	Suppose that $((K_i;\varphi_i))_{i \in I}, (q_i^j)_{i \leq j})$ is 
	a projective system and $(K;\varphi)$ together with extensions 
	$q_i \colon (K;\varphi) \rightarrow (K_i;\varphi_i)$ for $i \in I$ 
	is a projective limit. Then the corresponding invariant unital 
	$\mathrm{C}^*$-subalgebras $A_{i} \coloneqq T_{q_i}(\mathrm{C}(K_i))$ 
	for $i \in I$ (see \cref{fatopextensions}) satisfy
	\begin{enumerate}[(i)]
		\item $A_{i} \subset A_{j}$ for $i \leq j$, and 
		\item the union
		\begin{align*}
			\bigcup_{i \in I} A_i
		\end{align*}
		is dense in $\mathrm{C}(K)$,
	\end{enumerate}						 
  see \cite[Exercise 4.16]{EFHN2015}. Conversely, assume that 
  $(A_i)_{i \in I}$ is a net of invariant unital 
  $\mathrm{C}^*$-subalgebras $A_{i} \subset \mathrm{C}(K)$ for 
  $i \in I$ satisfying (i) and (ii). For every $i \in I$ we then 
  find an extension $q_i \colon (K;\varphi) \rightarrow (K_i;\varphi_i)$ 
  such that $A_{i} = T_{q_i}(\mathrm{C}(K_i))$ (see \cref{fatopextensions}) 
  and the canonical inclusion maps $A_{i} \hookrightarrow A_{j}$ for 
  $i \leq j$ induce extensions 
  $q_{i}^j \colon (K_j;\varphi_j) \rightarrow (K_i;\varphi_i)$ between 
  the associated systems. A moment's thought reveals that $(K;\varphi)$ 
  is a projective limit of the projective system 
  $((K_i;\varphi_i))_{i \in I}, (q_i^j)_{i \leq j})$. 
\end{remark}
The observations discussed in \cref{opprojective} are helpful for showing 
that a system is a projective limit of certain factors with specific 
properties. We demonstrate this 
by proving the following lemma which will be important soon.
\begin{lemma}\label{metrizablefact}
	Let $(K;\varphi)$ be a topological dynamical system. Then there are
	\begin{enumerate}[(i)]
  	\item an inductive system $((K_i;\varphi_i))_{i \in I}, (q_i^j)_{i \leq j})$ 
  	of metrizable systems, and
		\item extensions $q_i \colon (K;\varphi) \rightarrow (K_i;\varphi_i)$ for 
		every $i \in I$,
	\end{enumerate}
	such that $(K;\varphi)$ together with the extensions $q_i$ for $i \in I$ is a 
	projective limit of $((K_i;\varphi_i))_{i \in I}, (q_i^j)_{i \leq j})$.
\end{lemma}
\begin{proof}
	Let $I$ be the family of finite subsets of $\mathrm{C}(K)$ ordered by set 
	inclusion. For every $i \in I$ let $A_i$ be the invariant unital 
	$\mathrm{C}^*$-subalgebra generated by $i$. Then $A_i$ is separable for 
	every $i \in I$ and the net $(A_i)_{i \in I}$ satisfies properties (i) 
	and (ii) of \cref{opprojective}. By \cref{opprojective} we therefore find 
	a projective system $((K_i;\varphi_i))_{i \in I}, (q_i^j)_{i \leq j})$ and 
	extensions $q_i \colon (K;\varphi) \rightarrow (K_i;\varphi_i)$ for $i \in I$ 
	such that $(K;\varphi)$ is a projective limit of  
	$((K_i;\varphi_i))_{i \in I}, (q_i^j)_{i \leq j})$ and 
	$T_{q_i}(\mathrm{C}(K_i)) = A_i$ for every $i \in I$. Since $A_i$ is separable, 
	$K_i$ is metrizable for every $i \in I$ (see \cite[Theorem 4.7]{EFHN2015}), 
	which proves the claim.
\end{proof}
Let us now recall the famous Furstenberg structure theorem for minimal distal 
systems (see \cite[Chapter 7]{Ausl1988} and \cite[Section V.3]{deVr1993}).
\begin{theorem}\label{furstenberg}
	For a minimal system $(K;\varphi)$ the following assertions are equivalent.
	\begin{enumerate}[(a)]
		\item The system $(K;\varphi)$ is distal.
		\item There is an ordinal $\eta_0$ and a projective system 
		$((K_\eta;\varphi_\eta))_{\eta \leq \eta_0}, (q_\eta^\sigma)_{\eta \leq \sigma})$ 
		such that
	  \begin{enumerate}[(i)]
	  	\item $(K_1;\varphi_1)$ is a trivial system $(\{\mathrm{pt}\};\mathrm{id})$,
	  	\item $q_{\eta}^{\eta+1}$ is pseudoisometric for every $\eta < \eta_0$,
	  	\item $(K_\eta;\varphi_\eta) = \lim_{\gamma < \eta} (K_\gamma;\varphi_\gamma)$ 
	  	for every limit ordinal $\eta \leq \eta_0$.
	  \end{enumerate}
	\end{enumerate}
\end{theorem}

One can take part (b) of \cref{furstenberg} as an inspiration for the concept of 
measurably distal systems. To formulate this concept, we briefly recall the notion of 
inductive limits for measure-preserving systems (see \cite[Section 13.5]{EFHN2015}).

\begin{definition}
	Let $I$ be a directed set. For every $i \in I$, let $(\uX_i;T_i)$ be a 
	measure-preserving system and for $i \leq j$ let 
	$J_i^j \colon (\uX_i;T_i) \rightarrow (\uX_j;T_j)$ be an extension. 
	Suppose that 
	\begin{enumerate}[(i)]
		\item $J_j^k J_i^j = J_i^k$ for $i \leq j \leq k$, and
		\item $J_i^i = \mathrm{Id}$ for every $i \in I$.  
	\end{enumerate}
	Then the pair $(((\uX_i;T_i))_{i \in I}, (J_i^j)_{i \leq j})$ is an 
	\emph{inductive system}. 

  A measure-preserving system $(\uX;T)$ together with extensions 
  $J_i \colon (\uX_i;T_i) \to (\uX;T)$ such that $J_i = J_j J_i^j$ for 
  $i \leq j$ is an \emph{inductive limit of 
  $(((\uX_i;T_i))_{i \in I}, (J_i^j)_{i \leq j})$ } if it satisfies the 
  following universal property:
	\begin{itemize}
		\item Whenever $(\uY;S)$ is a measure-preserving system and 
		$I_i \colon (\uX_i;T_i) \to 
    (\uY;S)$ are extensions with $I_i =  I_j J_i^j$ for $i \leq j$, 
    then there is a unique extension $J 
    \colon (\uX;T) \to (\uY;S)$ such that the diagram
	\begin{align*}
		\xymatrix{
			& (\uY;S)\\
			(\uX_i;T_i)  \ar[ru]^{I_i} \ar[r]_{J_i}  & (\uX;T)\ar[u]_{J}  \\		
		}
  \end{align*}
		commutes for every $i$.
	\end{itemize}
	We then write 
	\begin{align*}
		(\uX;T) = \lim_{\substack{\longrightarrow\\ i}} (\uX_i;T_i).
	\end{align*}
\end{definition}
Every inductive system has an inductive limit 
(see \cite[Theorem 13.38]{EFHN2015}) and it is unique up to isomorphy.  
We now recall the definition of distal systems used by Furstenberg 
\cite[Definition 8.3]{Furs1977}.
\begin{definition}\label{defmeasurablydistal}
	A measure-preserving system $(\uX;T)$ is \emph{distal} if there is an 
	ordinal $\eta_0$ and an inductive system 
  $(((X_\eta;T_\eta))_{\eta \leq \eta_0}, (J_\eta^\sigma)_{\eta \leq \sigma})$ 
  such that
	\begin{enumerate}[(i)]
		\item $(\uX_1;T_1)$ is a trivial system $(\{\mathrm{pt}\};\mathrm{Id})$,
		\item $J_\eta^{\eta+1}$ has relatively discrete spectrum for every $\eta < \eta_0$,
		\item $(X_\eta;T_\eta) = \lim_{\mu < \eta} (X_\mu;T_\mu)$ for every limit 
		ordinal $\mu \leq \eta_0$.
	\end{enumerate}
\end{definition}
\begin{remark}
	If $\uX$ is a standard probability space, then there is an equivalent definition 
	in terms of so-called \emph{separating sieves}, see \cite{Parr1968} and 
	\cite[Theorem 8.7]{Zimm1976}.
\end{remark}
The measure-preserving system given by the skew-torus (see \cref{reldiscrex}) is a 
standard example for a distal measure-preservings system. By definition, it is 
obtained by equipping a topologically distal system with an invariant probability 
measure. Our main result, generalizing \cite[Theorem 4.4]{Lindenstrauss1999}, 
shows that, up to an isomorphism, every ergodic distal system can be obtained in 
this way. Moreover, the proof reveals a canonical choice for such a minimal distal model of a given distal ergodic measure-preserving system.

\begin{theorem}\label{mainthmdistal}
	Let $(\uX;T)$ be an ergodic distal measure-preserving system. Then there are a  minimal distal topological dynamical system $(K;\varphi)$ and a 
	fully supported 
	ergodic measure $\mu \in \mathrm{P}_{\varphi}(K)$ such that $(\uX;T)$ is isomorphic to $(K,\mu;T_\varphi)$. If $\uX$ is separable, 
	then $K$ can be (noncanonically) chosen to be metrizable.
\end{theorem}
The following lemma (cf. the proof of \cite[Theorem 4.4]{Lindenstrauss1999}) is the last missing ingredient for the proof of \cref{mainthmdistal}. 
\begin{lemma}\label{minimaldistal}
	If $(K;\varphi)$ is a distal topological dynamical system and there is a fully supported 
	ergodic measure $\mu \in \mathrm{P}_\varphi(K)$, then $(K;\varphi)$ is minimal.
\end{lemma}
\begin{proof}
	Assume first that $K$ is metrizable. Then the existence of a fully supported 
	ergodic measure guarantees the existence of a point $x \in K$ with dense orbit 
	$\{\varphi^n(x)\mid n \in \Z\}$, use Poincar\'{e} recurrence or Birkhoff's ergodic
	theorem (see \cite[Proposition 4.1.13]{KaHa1995}). But 
	then $(K;\varphi)$ is already minimal since a distal system decomposes into a 
	disjoint union of minimal systems, see \cite[Corollary 7]{Ausl1988}.

  If $K$ is not metrizable, we use \cref{metrizablefact} to write $(K;\varphi)$ 
  as a projective limit of metrizable factors $(K_i;\varphi_i)$ for $i \in I$.
  Since $(K_i;\varphi_i)$ is distal and admits a fully supported ergodic measure 
  (recall that the pushforward of an ergodic measure is again ergodic), we obtain 
  that $(K_i;\varphi)$ is minimal for every $i \in I$. Using that a projective 
  limit of minimal systems is minimal (see \cref{remproj}) we obtain that 
  $(K;\varphi)$ is itself minimal.
\end{proof}

\begin{proof}[Proof of \cref{mainthmdistal}]
  For an ergodic distal measure-preserving system $(\uX; T)$ take an ordinal $\eta_0$ and an inductive system 
	\begin{align*}
  	(((X_\eta;T_\eta))_{\eta \leq \eta_0}, (J_\eta^\sigma)_{\eta \leq \sigma})
	\end{align*}
  as in \cref{defmeasurablydistal}. Moreover, we write 
  $J_\eta \colon (\uX_\eta;T_\eta) \rightarrow (\uX;T)$ for the corresponding 
  extensions for every $\eta \leq \eta_0$. We now recursively construct 
	\begin{itemize}
		\item a projective system 
		$((K_\eta;\varphi_\eta))_{\eta \leq \eta_0}, (q_\eta^\sigma)_{\eta \leq \sigma})$,
		\item ergodic measures $\mu_{\eta} \in \mathrm{P}_{\varphi_\eta}(K_\eta)$ 
		for every $\eta \leq \eta_0$, and
		\item isomorphisms $\Phi_\eta \colon (K_\eta,\mu_\eta;T_{\varphi_\eta}) 
		\rightarrow (\uX_\eta;T_\eta)$ for every $\eta \leq \eta_0$,
	\end{itemize}
	such that $(q_\eta^\sigma,\mu_\sigma;\Phi_\eta,\Phi_\sigma)$ is a topological 
	model for $J_\eta^\sigma$ for all $\eta \leq \sigma \leq \eta_0$ and such that 
	$((K_\eta;\varphi_\eta))_{\eta \leq \eta_0}, (q_\eta^\sigma)_{\eta \leq \sigma})$ 
	is a projective system of minimal distal systems satisfying all the properties of 
	\cref{furstenberg} (b). From this the claim follows.

	Let $(K_1;\varphi_1)$ be a trivial system $(\{\mathrm{pt}\};\mathrm{id})$, 
	$\mu_1$ the unique probability measure on $K_1$ and 
	$\Phi_1 \colon (K_1,\mu_1;T_{\varphi_1}) \rightarrow (\uX_1;T_1)$ the 
	identity operator. Now assume that $\eta \leq \eta_0$ is an ordinal 
	and suppose we have already constructed $(K_\gamma;\varphi_\gamma)$ for every 
	$\gamma < \eta$, $q_\gamma^\sigma$ 
	for $\gamma \leq \sigma< \eta$, $\mu_{\gamma}$ for $\gamma < \eta$ and 
	$\Phi_\gamma$ for $\gamma < \eta$. We have to 
	consider two cases.
	\begin{enumerate}[(i)]		
  	\item Assume that $\eta$ is a successor ordinal, i.e., $\eta = \gamma +1$ 
  	for an ordinal $\gamma$. Since $(K_\gamma,\mu_{\gamma};T_{\varphi_{\gamma}})$ 
  	is isomorphic to $(X_\gamma;T_{\gamma})$ via $\Phi_\gamma$, we can apply 
  	\cref{mainthm} to find 
  	\begin{itemize}
  		\item a pseudoisometric extension 
  		$q_{\gamma}^{\eta} \colon (K_{\eta};\varphi_{\gamma}) \rightarrow 
  		(K_\gamma;\varphi_{\gamma})$,
	  	\item a fully supported ergodic measure 
	  	$\mu_{\eta} \in \mathrm{P}_{\varphi_{\eta }}(K_{\eta})$ with 
	  	$(q_{\gamma}^{\eta})_*\mu_{\eta} = \mu_{\gamma}$, and
		  \item an isomorphism $\Phi_{\eta} \colon 
		  (K_{\eta},\mu_{\eta};T_{\varphi_{\eta}}) \rightarrow (X_{\eta};T_{\eta})$,
	  \end{itemize}
	  such that $(q_{\gamma}^{\eta},\mu_{\eta};\Phi_{\gamma},\Phi_{\eta})$ is a 
	  topological model for $J_{\gamma}^\eta$. Since $(K_\gamma;\varphi_{\gamma})$ 
	  is distal and $q_{\gamma}^{\eta}$ is pseudoisometric, the system 
	  $(K_{\eta};\varphi_{\eta })$ is also distal. Moreover, 
	  $(K_{\eta};\varphi_{\eta})$ is minimal by \cref{minimaldistal}. We set 
	  $q_\sigma^{\eta} \coloneqq q_\sigma^\gamma\circ q_{\gamma}^{\eta} $ for 
	  every $\sigma < \gamma$.
	  \item If $\eta \leq \eta_0$ is a limit ordinal, we let 
	  $(K_{\eta};\varphi_{\eta})$ together with maps 
	  $q_\gamma^\eta \colon (K_\eta;\varphi_\eta) \rightarrow (K_\gamma;\varphi_\gamma)$ 
	  for $\gamma < \eta$ be a projective limit of the projective system 
	  $((K_\gamma;\varphi_\gamma))_{\gamma < \eta}, 
	  (q_\gamma^\sigma)_{\gamma \leq \sigma})$. Moreover, let $\mu_{\eta}$ 
	  be the ergodic measure on $(K_{\eta};\varphi_{\eta})$ induced by the net 
	  $(\mu_\gamma)_{\gamma < \eta}$ (cf. \cite[Exercise 10.13]{EFHN2015}), 
	  i.e., $\mu_{\eta} \in \mathrm{C}(K_\eta)'$ is uniquely determined by 
	  the identity $(q_\gamma^\eta)_{*}\mu_\eta = \mu_{\gamma}$ for every 
	  $\gamma < \eta$. By setting 
		\begin{align*}
			\Phi_{\eta}(T_{q_\gamma^\eta}f) \coloneqq J_\gamma^\eta \Phi_\gamma f
		\end{align*}
		for every $f \in \mathrm{C}(K_{\gamma})$ and $\gamma < \eta$ we obtain 
		a (well-defined) map 
		\begin{align*}
			\Phi_{\eta} \colon 
			\bigcup_{\gamma < \eta} T_{q_\gamma^\eta}(\mathrm{C}(K_{\gamma})) 
			\subset \mathrm{C}(K_\eta) \rightarrow \mathrm{L}^1(\uX_\eta)
		\end{align*}
		which extends to an isometric isomorphism 
		$\Phi_{\eta} \colon \mathrm{L}^1(K_\eta,\mu_\eta) 
		\rightarrow \mathrm{L}^1(\uX_\eta)$ intertwining the dynamics.  
	\end{enumerate}
	It is clear from the construction that 
	$(q_\eta^\sigma,\mu_\sigma;\Phi_\eta,\Phi_\sigma)$ is a topological model 
	for $J_\eta^\sigma$ for all $\eta \leq \sigma \leq \eta_0$.
	
	Finally, if $\uX$ is separable, then we can choose metric models in (i) 
	(see \cref{mainthm}). Moreover, in (ii) we can find a subsequence 
	$(((X_{\gamma_n};T_{\gamma_n}))_{n \in \N}, (J_{\gamma_{n}}^{\gamma_k})_{n \leq k})$ 
	of the projective system 
	$(((X_\gamma;T_\gamma))_{\gamma \leq \eta}, (J_\gamma^\sigma)_{\gamma \leq \sigma})$ 
	such that $(\uX_\eta;T_\eta)$ is still the inductive limit of that subsequence 
	(this is an easy consequence of the characterization (iii) of inductive limits 
	in \cite[Theorem 13.35]{EFHN2015}). By considering the now metrizable projective limit of
	$((K_{\gamma_n};\varphi_{\gamma_n}))_{n}, (q_{\gamma_n}^{\gamma_k})_{n \leq k})$ in 
	(ii) and then proceeding as before, we also obtain metrizable models in (ii).
\end{proof}
\begin{remark}\label{remcat}
	Our approach to the theorem of Lindenstrauss unveils a connection between topological and measure-preserving distal systems on a categorical level. Inspecting the definition of the canonical minimal distal model $\mathrm{Mod}(\uX;T) \coloneqq (K;\varphi)$ of an ergodic distal measure-preserving system $(\uX;T)$ in the proof of \cref{mainthmdistal}, shows that the assigment $(\uX;T) \mapsto \mathrm{Mod}(\uX;T)$ is actually functorial: Every extension $J \colon (\uY;S) \rightarrow (\uX;T)$ of ergodic distal measure-preserving systems induces an extension $\mathrm{Mod}(J) \colon \mathrm{Mod}(\uX;T) \rightarrow \mathrm{Mod}(\uY;S)$ between the corresponding canonical topological models. In this way, we obtain a (contravariant) functor $\mathrm{Mod}$ from the category of ergodic distal measure-preserving systems to the category of minimal distal topological dynamical systems. It is noteworthy however, that, even though we can also construct distal ergodic measure-preserving system from distal minimal systems (by simply choosing an ergodic invariant probability measure), the functor $\mathrm{Mod}$ does not define an equivalence between the two categories. In fact, if $(\uX;T)$ is an ergodic distal measure-preserving system and $(K;\varphi)$ its canonical model, then every eigenfunction of $T$ corresponds to a continuous eigenfunction of $T_\varphi$. With this observation one can readily show that a minimal distal system $(K;\varphi)$ possessing
	\begin{enumerate}[(i)]
		\item a unique invariant Borel probability measure $\mu$, and
		\item an eigenfunction $f \in \mathrm{L}^\infty(K,\mu)\setminus \mathrm{C}(K)$ with respect to $T_\varphi$
	\end{enumerate}
cannot be isomorphic to any canonical model $\mathrm{Mod}(\uX;T)$ of an ergodic distal measure-preserving system $(\uX;T)$. An example of Parry (see \cite[Section 3]{Parr1974}) demonstrates that such systems indeed exist and hence $\mathrm{Mod}$ does not define a categorical equivalence.
\end{remark}

\begin{remark}
	 
		In his article \cite{Lindenstrauss1999}, Lindenstrauss also discusses the question under which conditions an ergodic distal measure-preserving system has a distal model which is strictly ergodic (i.e., minimal with a unique invariant Borel probability measure), see also \cite{GutmanLian2019}. It is therefore an interesting problem to determine in which cases the canonical model constructed in this article is strictly ergodic.
	 
\end{remark}
\parindent 0pt
\parskip 0.5\baselineskip
\setlength{\footskip}{4ex}
\bibliographystyle{alpha}
\bibliography{./bib/bibliography} 

\begin{thebibliography}{EFHN15}

\bibitem[Aus88]{Ausl1988}
J.~Auslander.
\newblock {\em Minimal Flows and their Extensions}.
\newblock Elsevier, 1988.

\bibitem[Aus13]{Ausl2013}
J.~Auslander.
\newblock Characterizations of distal and equicontinuous extensions.
\newblock In I.~Assani, editor, {\em Ergodic Theory and Dynamical Systems.
  Proceedings of the Ergodic Theory Workshops at University of North Carolina
  at Chapel Hill, 2011-2012}, pages 59--66. De Gruyter, 2013.

\bibitem[BJM89]{BeJuMi1989}
J.~F. Berglund, H.~Junghenn, and P.~Milnes.
\newblock {\em Analysis on Semigroups. Function Spaces, Compactifications,
  Representations}.
\newblock Wiley, 1989.

\bibitem[Bro79]{Bron1979}
I.~U. Bron\v{s}te\v{i}n.
\newblock {\em Extensions of Minimal Transformation Groups}.
\newblock Sijthoff \& Noordhoff, 1979.

\bibitem[Der00]{Derr2000}
J.-M. Derrien.
\newblock On the existence of cohomologous continuous cocycles for cocycles
  with values in some {L}ie groups.
\newblock {\em Ann. Inst. Henri Poincaré Probab. Stat.}, 36:291--300, 2000.

\bibitem[DG83]{DuGi1983}
M.~J. Dupré and R.~M. Gillette.
\newblock {\em Banach Bundles, Banach Modules and Automorphisms of
  C*-Algebras}.
\newblock Longman, 1983.

\bibitem[dV93]{deVr1993}
J.~de~Vries.
\newblock {\em Elements of Topological Dynamics}.
\newblock Springer, 1993.

\bibitem[Ede19]{Edeko2019}
N.~Edeko.
\newblock On the isomorphism problem for non-minimal transformations with
  discrete spectrum.
\newblock {\em Discrete Contin. Dyn. Syst. Ser. A.}, 39(10):6001--6021, 2019.

\bibitem[EFHN15]{EFHN2015}
T.~Eisner, B.~Farkas, M.~Haase, and R.~Nagel.
\newblock {\em Operator Theoretic Aspects of Ergodic Theory}.
\newblock Springer, 2015.

\bibitem[EHK21]{EHK2021}
N.~Edeko, M.~Haase, and H.~Kreidler.
\newblock A decomposition theorem for unitary group representations on
  {K}aplansky-{H}ilbert modules and the {F}urstenberg-{Z}immer structure
  theorem.
\newblock arXiv:2104.04865, 2021.

\bibitem[EK21]{EdKr2020}
N.~Edeko and H.~Kreidler.
\newblock Uniform enveloping semigroupoids for groupoid actions.
\newblock {\em {T}o appear in J. Anal. Math.}, 2021.

\bibitem[Ell87]{Elli1987}
R.~Ellis.
\newblock Topological dynamics and ergodic theory.
\newblock {\em Ergodic Theory Dynam. Systems}, 7:25--47, 1987.

\bibitem[EW11]{EiWa2011}
M.~Einsiedler and T.~Ward.
\newblock {\em Ergodic Theory with a view towards Number Theory}.
\newblock Springer, 2011.

\bibitem[Fur63]{Furs1963}
H.~Furstenberg.
\newblock The structure of distal flows.
\newblock {\em Amer. J. Math.}, 85:477--515, 1963.

\bibitem[Fur77]{Furs1977}
H.~Furstenberg.
\newblock Ergodic behavior of diagonal measures and a theorem of {S}zemerédi
  on arithmetic progressions.
\newblock {\em J. Anal. Math.}, 31:204--256, 1977.

\bibitem[Gie82]{Gierz1982}
G.~Gierz.
\newblock {\em Bundles of Topological Vector Spaces and Their Duality}.
\newblock Springer, 1982.

\bibitem[GL19]{GutmanLian2019}
Y.~Gutman and Z.~Lian.
\newblock Strictly ergodic distal models and a new approach to the
  {H}ost--{K}ra factors, 2019.
\newblock arXiv:1909.11349v2.

\bibitem[Gla03]{Glas}
E.~Glasner.
\newblock {\em Ergodic Theory via Joinings}.
\newblock American Mathematical Society, 2003.

\bibitem[GW06]{GlWe2004}
E.~Glasner and B.~Weiss.
\newblock On the interplay between measurable and topological dynamics.
\newblock In B.~Hasselblatt and A.~Katok, editors, {\em Handbook of Dynamical
  Systems}, volume~1B, pages 597 -- 648. Elsevier, 2006.

\bibitem[HSX19]{HuangShaoXiangdong2019}
W.~Huang, S.~Shao, and Y.~Xiangdong.
\newblock Pointwise convergence of multiple ergodic averages and strictly
  ergodic models.
\newblock {\em J. Anal. Math.}, 139:265--305, 2019.

\bibitem[JT22]{JaTa2020}
A.~Jamneshan and T.~Tao.
\newblock An uncountable {M}ackey-{Z}immer theorem.
\newblock {\em Studia Math.}, 266:241--289, 2022.

\bibitem[KH95]{KaHa1995}
A.~Katok and B.~Hasselblatt.
\newblock {\em Introduction to the Modern Theory of Dynamical Systems}.
\newblock Cambridge University Press, 1995.

\bibitem[Lin99]{Lindenstrauss1999}
E.~Lindenstrauss.
\newblock Measurable distal and topological distal systems.
\newblock {\em Ergodic Theory Dynam. Systems}, 19(4):1063--1076, 1999.

\bibitem[Par68]{Parr1968}
W.~Parry.
\newblock Zero entropy of distal and related transformations.
\newblock In J.~Auslander and W.~H. Gottschalk, editors, {\em Topological
  Dynamics. An international symposium.}, pages 383--389. W. A. Benjamin, 1968.

\bibitem[Par74]{Parr1974}
W.~Parry.
\newblock A note on cocycles in ergodic theory.
\newblock {\em Compos. Math.}, 28:343--350, 1974.

\bibitem[Tao09]{Tao2009}
T.~Tao.
\newblock {\em Poincar\'{e}'s Legacies, Part I}.
\newblock American Mathematical Society, 2009.

\bibitem[Zim76]{Zimm1976}
R.~J. Zimmer.
\newblock Extensions of ergodic group actions.
\newblock {\em Illinois J. Math.}, 20:373--409, 1976.

\end{thebibliography}
\footnotesize
\end{document}